\documentclass[11pt, oneside,reqno]{amsart} 
\usepackage[showframe=false]{geometry}
\numberwithin{equation}{section}

\DeclareMathOperator*{\argmin}{arg\,min}

\newcommand*\diff{\mathop{}\!\mathrm{d}}

\newtheorem{theorem}{Theorem}[section]
\newtheorem{lemma}[theorem]{Lemma}

\theoremstyle{definition}
\newtheorem{remark}{Remark}[section]

\theoremstyle{plain}
\newtheorem{thm}{Theorem}[section]
\newtheorem{lem}[thm]{Lemma}

\theoremstyle{remark}
\newtheorem{rem}[thm]{Remark}
\theoremstyle{definition}
\newtheorem{defn}[thm]{Definition}
\theoremstyle{remark}
\newtheorem{assn}[thm]{Assumption}
\newtheorem{prob}[thm]{Problem}

\usepackage[retainorgcmds]{IEEEtrantools}
\newcommand{\s}{\mathcal{S}}
\newcommand{\rinf}{\mathcal{I}}

\newcommand{\E}{\mathbb{E}}

\newcommand{\ts}{\tilde{s}}

\def\tV{\tilde V}
\def\eqn{\begin{equation}}
\def\endeqn{\end{equation}}

\def\t{t}
\def\ct{\s}
\def\T{T}
\def\ts{{\tilde s}}
\def\hs{\hat s}
\def\mf{{m_f}}
\def\tmf{{\tilde m_f}}
\def\half{\frac{1}{2}}
\def\d{{\delta}}
\def\D{\Delta}
\def\fd{p_\d}
\def\G{\Gamma_\d}
\def\g{\gamma_\d}
\def\V{\mathbb{V}}
\def\M{\mathcal{M}}
\def\A{\mathcal{A}}
\def\C{\mathcal{C}}
\def\B{\beta}
\def\psd{\Psi^*_\d}
\def\ps{\Psi^*}
\def\pd{\Psi_\d}
\def\defto{:=}

\def\Om{\Omega}
\def\om{\omega}
\def\F{\mathcal{F}}
\def\R{\mathbb{R}}
\def\P{\mathbb{P}}
\begin{document}
\title{Dynamic minimisation of the commute time for a one-dimensional diffusion}
\author{Ma. Elena Hern\'andez-Hern\'andez$^1$}


\author{Saul D Jacka$^2$}
\thanks{\\
$^{1}$ School of Mathematics, University of Leeds, Leeds, LS2 9JT, UK\\Email: m.e.hernandez-hernandez@leeds.ac.uk
\\
$^{2}$ Department of Statistics, University of Warwick, Coventry CV4 7AL, UK\\ Email: s.d.jacka@warwick.ac.uk }

\begin{abstract}
Motivated in part by a problem  in simulated tempering (a form of Markov chain Monte Carlo) we seek to minimise, in a suitable sense, the time it takes a (regular) diffusion with instantaneous reflection at 0 and 1  to travel to $1$ and then return to the origin (the so-called commute time from 0 to 1). 
Substantially extending results in a previous paper, we consider a dynamic version of this problem where the control mechanism is related to the diffusion\rq{}s drift via the corresponding scale function.  We are only able to choose the drift at each point at the time of first visiting that point and the drift is constrained on a set of the form $[0,\ell)\cup(i,1]$. This leads to a type of stochastic control problem with infinite dimensional state.

\end{abstract}

\thanks{The authors are most grateful to Gareth Roberts for suggesting this problem}

\maketitle \centerline{{\today}}


\maketitle

\section{Introduction}
Suppose that $X^\mu$ is a diffusion on $[0,1]$,  started at 0 and given by
\begin{equation}\label{SDE}
 X^\mu_t=x+\int_0^t \sigma(X^\mu_u) \diff B_u+\int_0^t \mu(X^\mu_u) \diff u\,\,\,\text{ on $\,$ (0,1)}
\end{equation}
with instantaneous reflection at 0 and 1 (see \cite{R+W} or \cite{IM} for details).

Define $\T_x$ to be the first time that the diffusion
reaches $x$, then define $\ct=\ct(X^\mu)$, the commute time (between 0 and 1),
by
$$
\ct(X^\mu)\defto \inf\left \{ t>\T_1(X^\mu):\: X^\mu_t=0 \right \}.
$$

In \cite{JH}, motivated by a question arising in simulated tempering (see \cite{ARR}), we considered the following problem (and several variants and generalisations):
\prob\label{prob1} Minimise the expected commute time $\E [\, \ct\,]$; i.e. find
$$
\inf_\mu\E_x [\ct(X^\mu)],
$$
or, more generally,
find, for suitable positive functions $f$,
\begin{equation}\label{value}
\Phi(x)\defto\inf_\mu[\Phi(\mu,x)]\text{ where } \Phi(\mu,x)\defto\E_x[ \int_0^\ct f(X^\mu_t) \diff t],
\end{equation}
and where the drift at each level must be chosen at or before $X^\mu$'s first visit to that level and thereafter remains fixed.
\endprob
The commute time is defined  for random walks on graphs in \cite{Barlow}. The original commute time identity (a version of which we give later in (\ref{ctidentity})), was only discovered in 1989 and first appeared in \cite{CRRST}.
We gave the optimal drift to minimise the quantity in (\ref{value}) in Theorem 4.6 of \cite{JH}, under the assumption that $\mu$ was already fixed on some interval $[0,y)$, the drift is otherwise unconstrained and the starting state is in $[0,y)$. We left open the question of the optimal control when $\mu$ is, initially, fixed on some other interval.
The key observation in \cite{JH} was that we can follow the same solution as for the static case -- where we choose the drift function at time 0 -- because \lq there can be no surprises' (in the path of $X$). This statement is no longer valid when the set on which the drift is constrained is not of the form $[0,y)$ and we gave,  as an example in Remark 4.1 of \cite{JH}, a heuristic argument for why a different solution would be optimal in the case where the constraint set was of the form $(i,1]$.

Our aim, in the current paper, is to present the solution (in Theorem  \ref{result}) to the dynamic problem in this case, where the \lq surprises' are how far down the controlled process gets before time $\T_1$.

The structure of the paper is as follows: in section 2 we give a formal definition of the model; section 3 is devoted to calculating the candidate value function; section 4 gives the proof that this is correct and we then give some concluding remarks.
\section{The model and some notation}
\subsection{The model}
Let $X_{t}^{\mu,i_0}$ be a regular diffusion on $[0,1]$ with instantaneous reflection at 0 and 1, starting at $i_0\in[0,1]$, and given by
\[
X_{t}^{\mu,i_0} =i_0+\int_0^t \sigma(X_{t}^{\mu,i_0})\diff B_{t}  + \mu (X_{t}^{\mu,i_0})\diff s.
\]


We need to define the set of admissible controls quite carefully. Two approaches are possible: the first is to restrict controls to choosing the drift $\mu$  whilst the second is to control the corresponding random scale function.We adopt the second approach, although we should caveat that the identified optimal policy will not, in general, be in this class (or, equivalently, the relevant infimum will not be attained by any policy in this class).

We assume the usual Markovian setup, so that each stochastic process lives on a suitable filtered space $(\Om,\F,(\F_t)_{t\geq 0})$, with the usual family of probability measures $(\P_x)_{x\in[0,1]}$ corresponding to the possible initial values.

Let $X^{\mu}$ be the diffusion with instantaneous reflection as given in \eqref{SDE}. Denote by $s^\mu$ the standardised \textit{scale measure} of   $X^\mu$  and by $m^\mu$ the corresponding \textit{speed measure}. 
\begin{remark}
Since $X^\mu$ is regular and reflection is instantaneous we have:
$$
s^\mu(x)\defto s^\mu[0,x] = \int_0^x\exp\biggl( {-2\int_0^u\frac{\mu(t)}{\sigma^2(t)} \diff t}\biggr) \diff u,
$$
$$
m^\mu([0,x])\defto m^\mu(x)=2\int_0^x\frac{ \diff u}{\sigma^2(u)s\rq{}(u)}=2\int_0^x\frac{\exp \bigl({2\int_0^u\frac{\mu(t)}{\sigma^2(t)} \diff t}\bigr)}{\sigma^2(u)} \diff u,
$$
(see \cite{RY}).
\end{remark}
From now on, we  consider the more general case where we only know that 
$s$ and $m$ are absolutely continuous with respect to Lebesgue measure (denoted by $\lambda$) so that, denoting the respective Radon-Nikodym derivatives by $s\rq{}$ and $m\rq{}$ we have
$$
s\rq{}m\rq{}= \frac{2}{\sigma^2}\,\,\text{ $\lambda$-a.e.}
$$
For such a pair we shall denote the corresponding diffusion, when it exists, by $X^s$. We emphasize that we are only considering regular diffusions with Brownian \lq\lq{}martingale  part\rq\rq{} $\int\sigma \diff B$ or, more precisely, diffusions $X$ with scale functions $s$ such that
\begin{equation}\label{model}
 \diff s(X_t)=s'(X_t)\sigma(X_t) \diff B_t,
\end{equation}
for some Brownian Motion $B$, 
so that, for example, sticky points are excluded (see \cite{E} for a description of driftless sticky Brownian Motion and its construction, see also \cite{E2} for other problems arising in solving stochastic differential equations ) as are singular scale measures. 
\begin{remark}
Note that our assumptions do allow generalised drift: if $s$ is the difference between two convex functions (which we will not necessarily assume) then 
\eqn\label{gend}
X^s_t=x+\int_0^t\sigma(X^s_u) \diff B_u-\half\int_{\R}L^a_t(X)\frac{s\rq{}\rq{}( \diff a)}{s\rq{}_-(a)},
\endeqn
where $s\rq{}_-$ denotes the left-hand derivative of $s$,  $s\rq{}\rq{}$ denotes the signed measure induced by $s\rq{}_-$ and $L^a_t(X)$ denotes the local time at $a$ developed by time $t$ by $X$ (see \cite{RY} Chapter VI for details).
\end{remark}
\begin{remark}
Essentially, we treat (\ref{model}) as the canonical dynamics for our problem, but note that we shall consider random scale measures for which $s(0)$ is not known at time 0. 
\end{remark}
\subsection{The control setting}
As is usual, we will adopt a weak approach to the control problem so that we work on canonical pathspace $\Om=D_{[0,\infty)}[0,1]$, the space of paths in $[0,1]$ which are right-continuous with left limits indexed by $[0,\infty)$ equipped with the Borel $\sigma$-algebra on $\Om$ with respect to the Skorokhod metric and natural filtration $(\F_t)_{t\ge 0}$  with
$$
X:\om\mapsto\om
$$
(see \cite{EK} for details).

Let $\mathcal{M}$ be the set of scale functions/measures on $[0,1]$ that are absolutely continuous with respect to Lebesgue measure.  Given a fixed scale function $s_0 \in \mathcal{M}$  and a Borel subset $F$  of $[0,1]$, define the set $\mathcal{M}_{F}^{s_0}$ as follows
\begin{align}
\M_{F}^{s_{0}} \defto  \left \{ s  \in \mathcal{M}\; a.s. : \,\,\diff s \Big |_F = \diff s_0  \Big |_F\; \right \}.
\end{align}
Then define the set of available controls
$$
\C_{F}^{s_{0}} \defto  \left \{\text{random measures } s  \in \mathcal{M}\; a.s. : \,\,\diff s \Big |_F = \diff s_0  \Big |_F,\; s'(X_{\cdot})\text{ is predictable}\right \}.
$$
Now define the admissible control policies $\A_{F}^{x,s_{0}}$ to be the set of $s\in\C_{F}^{s_{0}}$ such that the corresponding controlled process starting at $x \in [0,1]$ with random scale function $s$ exists; 
in other words, 
\begin{IEEEeqnarray}{r,l}\label{adm}
\A_{F}^{x,s_{0}}=&\{s\in\C_{F}^{s_{0}}\text{ s.t.  there exists a p.m. }\P_{s,x}:\:\: 
\text{ with }X_0=x\;\P_{s,x}-a.s.\text{ and }\\
&\phantom{xxx}ds(X_t,\om)=s'(X_t,\om)\sigma(X_t)dB_t,
\text{ for some }\P_{s,x}-\text{Brownian Motion},\; B\},
\end{IEEEeqnarray}
We denote the expectation corresponding to such a p.m. $\P_{s,x}$ by $\E_{s,x}$.

Recall that $\T_y$ is the first hitting time at level $y$ of the process $X$, that is
\[ \T_y := \inf \{ t > 0\,:\, X= y\},\]
and $\s$ denotes the first time the controlled process  reaches $0$ after having hit level $1$, that is 
\begin{equation}
\s: = \inf \{ t > \T_1 \,:\, X_{t} = 0\}.
\end{equation}

Define the running infimum of $X$ by setting
\[ \rinf_{t} \,\, :=\,\,\inf_{0\le r \le t} X_{r}.\]

We are able to choose $s'$ (the derivative of the scale function $s$)  dynamically, but {\em only once for each level}, and we seek to minimise. 
\[ \mathbb{E} \left [ \int_{0}^{\s} f (X_{t}) \diff t\right ]\]
(See  \cite[Remark 4.1]{JH}).

More precisely, we will assume the following:
\begin{assn}\label{A1}For a given level $\ell$ and a starting point $i$, with $\ell < i$,  suppose that $s'$ has been fixed at every level on $C:=[0,\ell] \cup [i,1]$.
\end{assn}

For a given positive cost function $f$, the control problem consists in finding
\begin{equation}
V(i)=\V(s_0,C,i)\defto \inf_{s\in \A_{C}^{i,s_0}} \mathbb{E}_{s,i} \int_0^{\s} f(X_{t})\diff t.
\end{equation}

\subsection{Heuristic for the optimal strategy.} 
Strategy $s$ has been fixed on $C_0:=[0,\ell)\cup (i_0,1]$, and we need to determine how to proceed on $[\ell,i_0]$. Since we are only allowed to choose $s'$ once for each level, we choose a strategy at $y\in (\ell, i_0)$ before time $\T_1$ only if such a level is reached from above before hitting $1$. Conversely, if $\rinf_{\T_1}>y$, then we need only choose the drift at level $y$ after the process has hit level 1. Consequently we are not constrained to enable the process to hit level 1 (again) so may choose an arbitrarily large downward drift at any such level.

Our optimal control should respect this and we proceed to calculate the optimal control in this class.
\section{The candidate optimal payoff}
\subsection{Initial calculations}
The commute time identity:
\begin{equation}\label{ctidentity}
\E_0\s=s[0,1]m[0,1]
\end{equation}
is generalised as follows.
For any pair $x,y \in [0,1]$, and any $s\in\M$, define the function
\[\phi^s(x,y)=\phi (x,y) := \E_{s,x} \left [ \int_{0}^{\T_{y}} f(X_{t}) \diff t\right ]\]
It follows (by \cite[Theorem 2.4]{JH}) that,
defining the measure $m_f$ by
$$
m_f'=fm',
$$
\begin{equation}\label{phi}
\phi(x,y) = \left \{
\begin{array}{cc}
\int_{x}^{y} \diff v s'(v) \int_{0}^{v} \frac{2f(u)}{\sigma^{2}(u) s'(u)} \diff u=\int_{x}^{y} \diff v s'(v) m_f(v), & x \le y \\
&\\
\int_{y}^{x} \diff v s'(v) \int_{v}^{1} \frac{2f(u)}{\sigma^{2}(u) s'(u)} \diff u=\int_{y}^{x} \diff v s'(v)\tmf(v), & x \ge y,
\end{array}
\right.
\end{equation}
and
$$\E_{s,0}\int_0^\s f(X_t)dt=\phi(0,1)+\phi(1,0)=s[0,1]m_f[0,1].
$$

Now we suppose that $s$ is an arbitrary  (deterministic) scale function which is assumed to equal $s_0$ on the intervals $[0,\ell]$ and $[i,1]$ (for some $\ell \in [0,1]$ fixed in advance) and which will be dynamically reset to give infinite downward drift on the interval $[\ell,I_{\T_1}\vee\ell)$ at time $\T_1$. We denote the corresponding (random) scale function by $s^*$.

Define 
$$
\rho:z\mapsto  \frac{2f(z)}{\sigma^2(z)}.
$$
We make the following standing assumption whose relevance follows:
\begin{assn}\label{cond}
$$
\sqrt{\rho}\in L^1[0,1].
$$
\end{assn}
\begin{rem}
Suppose that $s$ is a scale measure,
then the Cauchy-Schwarz inequality tells us that
$$
\int_0^1\rho(z)dz=\int_0^1\frac{2f(z)}{\sigma^2(z)}dz\leq \biggl(\int_0^1 s'(z)dz\biggr)^\half\biggl(\int_0^1 \frac{2f(z)}{\sigma^2(z)s'(z)}dz\biggr)^\half=\sqrt{\Phi(s,0)},
$$
so that Assumption \ref{cond} is a necessary condition for the existence of a scale function with finite payoff.

\end{rem}

We compute the corresponding payoff of a controlled process $\{X_{t}\}_{t\ge 0}$, where $i$ is the starting point.

Set  $s:x\mapsto s[0,x]$, $m_f:x\mapsto m_f[0,x]$, $\ts:x\mapsto s[x,1]$, and $\tmf:x\mapsto m_f[x,1]$

\begin{defn}Fix $s\in\M_{C}^{s_0}$ and define
\begin{IEEEeqnarray}{rl}\label{param}
\kappa &= \int_0^\ell s'(v)m_f [v,\ell]\diff v=\int_0^l s'(v)\tmf (v)\diff v-s(\ell)\tmf(\ell)\nonumber\\
a &= \tilde{m}_f (i) \nonumber\\
b &= \tilde{s} (i) \nonumber\\
c &=m_f(\ell)\nonumber\\
\t &= \tilde{s} (\ell) \nonumber\\
k &= s (\ell),
\end{IEEEeqnarray}
and note that, since $s'$ is fixed on $C_0:=C=[0,\ell)\cup (i,1]$, $t$ is the only one of the parameters in (\ref{param}) which is not determined by $s'$ restricted to $C_0$.
\end{defn}
To complete the (infinite-dimensional) state of the problem, we define
$$
C_t\defto C_0\cup(\rinf_{t\wedge\T_1},i]=[0,\ell)\cup(\rinf_{t\wedge\T_1},1].
$$

\begin{lem}
Let $s\in \M$, then the  payoff of the policy $s^*$ is given by
\begin{equation}\label{vstar}
V(s^*,C_0,i)=V^{s^*}(i)\defto \left \{ \kappa + bc + ab H(i) + \frac{1}{e^{H(i)}} \int_l^i\frac{\rho H e^H }{H'} \diff z \right \}
\end{equation}
where $H$ is defined by
\begin{equation*}
H:z\mapsto  1 + \frac{\t}{c} + \ln \frac{c}{\tilde{s}(z)}\text{ and, as stated earlier, }
\rho:z\mapsto  \frac{2f(z)}{\sigma^2(z)}.
\end{equation*}
\end{lem}
\begin{proof}
First note that $s^*=s$ on the event $(\rinf_{\T_1}\leq \ell)$, whereas, on $(\rinf_{\T_1}>\ell)$,  $s^{*'}=s'$ on $(\rinf_{\T_1},1]$ and when $X^{s^*}$ reaches $\rinf_{\T_1}$ for the first time after time $\T_1$, the effect of infinite downwards drift is that $X$ is instantaneously translated to level $\ell$, and thereafter a reflecting (downwards) barrier is imposed at level $\ell$.

It follows that
\begin{equation}\label{cost1}
V^{s^*}(i)=\phi^s(i,1)+\E_i[\phi^s(1,0)1_{(\rinf_{\T_1}\leq \ell)}+\bigl(\phi^s(1,\rinf_{\T_1})+\phi^s_{\ell}(\ell,0)\bigr)1_{(\rinf_{\T_1}> \ell)}],
\end{equation}
where $\phi^s_{\ell}(x,y):=\mathbb{E} \left [ \int_{0}^{\T_{y}} f(X_{t}^{x,s}) 1_{(X_{t}^{x,s}\in[0,\ell])}\diff t\right ]=\mathbb{E} \left [ \int_{0}^{\T_{y}} f(X_{t}^{(l),x,s})\diff t\right ]$
and $X^{(l),x}$ is the controlled process started at $x$ and with a reflecting barrier at $\ell$.

It is easy to see that
$$
\phi^s_{\ell}(x,y) = \int_{y}^{x} \diff v s'(v) \int_{v}^{\ell} \frac{2f(u)}{\sigma^{2}(u) s'(u)} \diff u=\int_{x}^{y} \diff v s'(v) m_f(v,\ell), \text{ for } y \le x \le \ell,
$$
so that $\phi^s_{\ell}(\ell,0)=\kappa$, while, under the control $s$, with $X$ starting at $i$, the distribution function of $\rinf_{\T_1}$ is 
\begin{equation}\label{dist}
F:x\mapsto \begin{cases}
1&\text{for }x>i\\
\frac{\ts(i)}{\ts(x)}&\text{for }i\geq x\geq 0\\
0&\text{for }x<0.
\end{cases}
\end{equation}
Now equation (\ref{cost1}) implies that
\[
V^{s^*}(i)=\phi^s(i,1)+\phi^s(1,0)F(\ell)+[1-F(\ell)]\kappa+\int_{\ell}^i\phi^s(1,x)dF(x),
\]
which becomes (on integrating by parts and recalling (\ref{phi} and (\ref{dist}))
\begin{IEEEeqnarray*}{rl}\label{cost2}
V^{s^*}(i)&=\int_i^1 s'(v)m_f(v)\diff v+\frac{\ts(i)}{\ts(\ell)}\int_0^1 s'(v)\tmf(v)\diff v+\kappa\biggl(1-\frac{\ts(i)}{\ts(\ell)}\biggr)\\
&\phantom{space}+\int_{\ell}^i \frac{\ts(i)}{\ts(x)}s'(x)\tmf(x)dx+F(i)\phi^s(1,i)-F(\ell)\phi^s(1,\ell)\\
&=\int_i^1 s'(v)m_f(v)\diff v+\kappa+
\int_{\ell}^i \frac{\ts(i)}{\ts(x)}s'(x)\tmf(x)dx\\
&\phantom{space}+\int_i^1 s'(v)\tmf(v)\diff v-\frac{\ts(i)}{\ts(\ell)}(\int_{\ell}^1 s'(v)\tmf(v)\diff v+\kappa-\int_0^1 s'(v)\tmf(v)\diff v)\\
&=\int_i^1 s'(v)m_f(v)\diff v+\kappa+
\int_{\ell}^i \frac{\ts(i)}{\ts(x)}s'(x)\tmf(x)dx+\int_i^1 s'(v)\tmf(v)\diff v+\frac{\ts(i)}{\ts(\ell)}s(\ell)\tmf(\ell)\\
&=\ts(i)m_f([0,1])+\kappa+\int_i^1 s'(v)\tmf(v)\diff v+\frac{\ts(i)}{\ts(\ell)}s(\ell)\tmf(\ell)\\
&=\ts(i)m_f(\ell)+\kappa+\ts(i)\biggl[(1+\frac{s(\ell)}{\ts(\ell)})\tmf(\ell)+\int_\ell^i\frac{\tmf(y)}{\ts(y)}s'(y)\diff y   \biggr]\\
&=\ts(i)m_f(\ell)+\kappa+\ts(i)\biggl[(1+\frac{s(\ell)}{\ts(\ell)})\tmf(\ell)+\int_\ell^i \tmf(y)H'(y)\diff y   \biggr]\\
&=\ts(i)m_f(\ell)+\kappa+\ts(i)\biggl[H(i)\tmf(i)+\int_\ell^i m_f'(y)H(y)\diff y\biggr]\\
&=\kappa+bc+abH(i)+b\int_\ell^i m_f'(y)H(y)\diff y.\\
\end{IEEEeqnarray*}
Now $m_f'=\frac{\rho}{s'}$ so, using the equalities
\begin{align*}
H'(z) = \frac{s'(z)}{\tilde{s}(z)}, \text{ and } e^H = \frac{c}{\tilde{s}(z)} e^{1+ \t/c},
\end{align*}
we obtain
$$
V^{s^*}(i)=\kappa+bc+abH(i)+ \frac{1}{e^{H(i)}} \int_l^i\frac{\rho H e^H }{H'} \diff z, 
$$
as required.
\end{proof}

\subsection{A calculus of variations approach}
We wish to find
\begin{equation}\label{prob1}
V(i):=\inf_{s\in\M_{C_0}^{s_0}}V^{s^*}(i).
\end{equation}
since our candidate optimal control lies in this class.
To do this, we first treat $H(\ell)$ and $H(i)$ as fixed parameters and then optimise over suitable values for these parameters.
\begin{lem}
\[V(i) := \kappa + bc + \inf_{\t\ge b} \left \{  ab \left (1+ \frac{k}{\t} + \ln  \frac{\t}{b}\right) + \frac{\B^2(i)}{e^{1+ \frac{k}{\t} + \ln  \frac{\t}{b}} \bigl(\phi(1+\frac{k}{\t})-\phi(1+ \frac{k}{\t} + \ln  \frac{\t}{b})\bigr)}\right \}
\]
where
\[\phi:x\mapsto  \int_{x}^{\infty} \frac{\diff u}{ue^u},\quad\quad\B (i) = \int_{\ell}^i \sqrt{\rho(v)}\diff v.\]
\end{lem}
\begin{proof}
Fixing $H(\ell)$ and $H(i)$, the Euler-Lagrange equation for the minimisation of $\int_{\ell}^i\frac{\rho H e^H }{H'} \diff z $ yields
$$
(2\rho H'(1+H)+\rho' H)H'-2H''\rho H=0.
$$
Dividing by $2\rho HH'$ gives
\begin{equation}\label{EL1}
H'(1+\frac{1}{H})=\frac{H''}{H'}-\half\frac{\rho'}{\rho}.
\end{equation}
Integrating (\ref{EL1}) yields
$$
H+\ln H=\ln H'-\half\ln\rho-\ln D,
$$
or
\begin{equation}\label{EL2}
\frac{H'}{He^H}={D\sqrt{\rho}},
\end{equation}
for some constant of integration $D$.
Integrating (\ref{EL2}) from $\ell$ to $y$ yields
\begin{equation}\label{EL3}
D\B(y)=\int_{H(\ell)}^{H(y)}\frac{du}{ue^u}=\phi(H(\ell))-\phi(H(y)),
\end{equation}
and applying the boundary condition at $i$ yields
\begin{equation}\label{EL3}
D=\frac{\phi(H(\ell))-\phi(H(i))}{\B(i)}.
\end{equation}

Substituting (\ref{EL2}) into (\ref{vstar}) gives 
\begin{IEEEeqnarray*}{rl}
V(i)&=\kappa + bc + \inf_{H(i)}\biggl(ab H(i) + \frac{\B(i) }{De^{H(i)}}\biggr)\\
&=\kappa+bc+\inf_{H(i)}\biggl(abH(i)+\frac{\B^2(i) }{[\phi(H(\ell))-\phi(H(i))]e^{H(i)}}\biggr)\\
&=\kappa + bc + \inf_{\t\geq b}\biggl(  ab \left (1+ \frac{k}{\t} + \ln  \frac{\t}{b}\right) + \frac{\B^2(i)}{e^{1+ \frac{k}{\t} + \ln  \frac{\t}{b}} \bigl[\phi(1+\frac{k}{\t})-\phi(1+ \frac{k}{\t} + \ln  \frac{\t}{b})\bigr]}\biggr)
\end{IEEEeqnarray*}
as required.
\end{proof}

\subsection{Further optimisation}
Define 
$$\delta=\delta(i)=\frac{k}{b}=\frac{s(\ell)}{\ts(i)}$$
and $$\fd:y\mapsto 1+\frac{\d}{y}+\ln y,
$$
Note that $\argmin(\fd)=\d$ and, defining $\G=\max(\d,1)$, $\fd$ is decreasing on $[1,\G]$ and increasing on $[\G,\infty)$, with infimum $\g=\min(1+\d,2+\ln\d)$. 
Define $\fd^{-1}:[\g,\infty)\rightarrow [\G,\infty)$ then
the change of variable $ z = \fd(\frac{\t}{b}) := 1+\frac{k}{\t} + \ln (\frac{\t}{b})$ yields
\[V(i) := \kappa + \tilde{s}(i) m_f(l) + ab + \inf_{z\geq \g} \left \{  ab z + \frac{\beta^2 (l,i)e^{-z}}{\phi(z-\ln f^{-1}(z))-\phi(z)}\right \}.
\]

Define 
\begin{equation}\label{D:Psi}
 \Psi_\d (z) : = \frac{e^{-z}}{\phi(z-\ln \fd^{-1}(z))-\phi(z)}.
 \end{equation}
 For now, we shall omit the dependence of $\pd$ on $\d$.
 
\begin{lemma} $\Psi$ is a positive, decreasing convex function on $(\g,\infty)$. 
\end{lemma}
\begin{proof}
Since $\phi$ is strictly decreasing, $\Psi$ is finite and positive on $(\g,\infty)$.

Differentiating, we obtain
\begin{IEEEeqnarray}{rl}\label{psid}
\Psi'(x)&=-\biggl(\Psi(x)+e^x\Psi^2(x)[\phi'(x-\ln \fd^{-1}(x))(1-\frac{(\fd^{-1}(x))')}{\fd^{-1}(x)}-\phi'(x)]\biggr)\nonumber\\
&=-\biggl(\Psi(x)+e^x\Psi^2(x)\biggl[\frac{1}{xe^x}+\frac{\d}{\fd^{-1}(x)-\d}\frac{e^{\ln \fd^{-1}(x)-x}}{x-\ln \fd^{-1}(x)}\biggr]\biggr)\nonumber\\
&=-\biggl(\Psi(x)+\Psi^2(x)\biggl[\frac{1}{x}+\frac{\d}{\fd^{-1}(x)-\d}\frac{\fd^{-1}(x)}{x-\ln \fd^{-1}(x)}\biggr]\biggr)\nonumber\\
&=-\biggl(\Psi(x)+\Psi^2(x)\biggl[\frac{1}{x}+\frac{\d\fd^{-1}(x)^2}{\fd^{-1}(x)^2-\d^2}\biggr]\biggr)\text{ (since }x-\ln\fd^{-1}(x)=1+\frac{\d}{\fd^{-1}(x)})\\
&<0\text{ since $\Psi>0$ and $\fd^{-1}\geq\G\ge\d$}\nonumber
\end{IEEEeqnarray}
Observe that the function $g:x\mapsto\frac{1}{x}+\frac{\d\fd^{-1}(x)^2}{\fd^{-1}(x)^2-\d^2}$ appearing in (\ref{psid}) is positive and decreasing on $[\g,\infty)$ since $\fd^{-1}$ is increasing and greater than $\d$ on $[\g,\infty)$.

Finally, from (\ref{psid}), 
\begin{equation}\label{g}
\psi'(x)=-\Psi(x)-\Psi^2(x)g(x)
\end{equation}
 and so
$$
\psi''(x)=-\Psi'(x)(1+2\Psi(x)g(x))-\Psi^2(x)g'(x)>0\text{ since $\Psi,g>0$, $\Psi'<0$ and $g'\le 0$},
$$
establishing that $\Psi$ is strictly convex. 
\end{proof}

We denote the convex conjugate of $\Psi$ by $\Psi^*_\d$,or just $\Psi^*$, so that
$$
\Psi^*:x\mapsto \inf_{z\geq \g}\biggl(zx+\Psi(z)\biggr).
$$

We have established the following:
\begin{lem}
Define 
\[ r :=  \frac{ab}{\B^2 (i)}  \]
Then, $V(i)$ is given by
\begin{align*}
V(i) :&= \kappa + bc+\B^2(i)\Psi^*( r).
\end{align*}
\end{lem}
\section{The dynamic solution}
Adopting the usual dynamic approach, we want to give the (conditional) future optimal payoff when  $\rinf_t=i$ and $X_t=x$ (with $x\geq i$) (and the process has not yet reached level 1). 

Since we cannot control the process until it reaches $\rinf_t$ again, the payoff is given by 
\begin{align*}
V(s,i,x) = V(i,x)=\mathbf{E}_{s,x} \int_0^{\T_1 \wedge \T_i} f(X_u)\diff u + \frac{\tilde{s} (x)}{\tilde{s} (i)} V(i) + \left (1 - \frac{\tilde{s} (x)}{\tilde{s} (i)} \right) \tV(x),
\end{align*}
where
\begin{align*}
\mathbf{E}_x \int_0^{\T_1 \wedge \T_i} f(X_u)\diff u &= -\int_x^1 \diff u s'(u) \tilde{m}_f (u) + \frac{\tilde{s} (x)}{\tilde{s} (i)} \int_i^1\diff u s'(u)\tilde{m}_f (u)\\
\text{and }\tV(x) &= \mathbb{E}_{s,1}\int_0^{\T_i} f(X_u) \diff u + \mathbb{E}_l^{*l} \int_0^{\T_0} f(X_u)\diff u.
\end{align*}
Observe that $\kappa = \mathbb{E}_l^{*l} \int_0^{\T_0} f(X_u)\diff u$.

Equivalently, we can rewrite this as
\begin{align*}
V(i,x) &= \mathbf{E}_{s,x} \int_0^{\T_i } f(X_u)\diff u + \frac{\tilde{s} (x)}{\tilde{s} (i)} V(i) + \left (1 - \frac{\tilde{s} (x)}{\tilde{s} (i)} \right) \kappa \\
&=  \int_i^x \diff u s'(u) \tilde{m}_f (u)+ \frac{\tilde{s} (x)}{\tilde{s} (i)} V(i) + \left (1 - \frac{\tilde{s} (x)}{\tilde{s} (i)} \right) \kappa\\
&=  \int_i^x \diff u s'(u) \tilde{m}_f (u)+ \frac{\tilde{s} (x)}{\tilde{s} (i)} ( V(i) - \kappa) + \kappa
\end{align*}

As is usual, to show that $V$ really is the optimal payoff, we consider the 
processes $S^s$ corresponding to using the policy $s$ until time $t$ and then behaving optimally.
What is non-standard here is the \lq phase change' at time $\T_1$ --- after $\T_1$, $\s$ just looks like $\T_0$.

Consequently, we enlarge the state by including a flag process $F$: 
$$
F_t=0\text{ on }[0,\T_1)\text{ and }F_t=1\text { on }[\T_1,\s],
$$
so that the generic state becomes $(s,X_t,\rinf_t\vee \ell,F_t)$ (or more precisely $(s|_{C_t},C_t,X_t,F_t)$ where $C_t=[0.\ell)\cup(\rinf_t,1]$),
and define
\begin{equation}\label{bell}
S^s_t=\int_0^{\s\wedge t} f(X_u)du +V(s,C_t,X_{\s\wedge t})1_{(t<\T_1)}+\tV(s,C_{{\T_1}}X_{\s\wedge t})1_{(t\ge \T_1)},
\end{equation}
where $\tV$ is our proposed \lq post-$\T_1$' payoff given by
\begin{equation}\label{bell2}
\tV(s,C,x)=\tV(x)\defto\begin{cases}
\int_i^x s'(v)\tmf(v)dv+\kappa&\text{ for }x\ge i\\
\kappa&\text{ for }x\in (\ell,i]\\
\int_0^x s'(v)m_f(v,\ell)dv&\text{ for }x\in[0,\ell].
\end{cases}
\end{equation}

\begin{thm}\label{sub}
Suppose that $(S^s_{t\wedge\s})_{t\geq 0}$ is a $\P_{s,i}$-submartingale for any admissible policy $s$ and 
\begin{itemize}
\item[1.] $\E_{s,x,i,0} [S^{\hs}_{\T_1}]=V(s,x,i)$  for some admissible policy $\hs$ with $\hs\in\A^{i,s_0}_{C}$;
\item[2.] $\liminf \E_{\hs^n,x,i,1} [S^{\hs^n}_{\s}]=\tV(s,C,x)$,
 for some sequence of admissible policies $(\hs^n)_{n\geq 1}$ with $\hs^n\in\A^{i,s_0}_{C}$
\end{itemize}
 then 
$$
\V(s,C,x,f)=\begin{cases}
V(s,C,x):&f=0\\
\tV(s,C,x):&f=1,
\end{cases}
$$ 
and $\hs^*$ is an optimal policy.
\end{thm}
\begin{proof}
Suppose that $s$ is admissible and  $(S^s_{t\wedge\s})_{t\geq 0}$ is a submartingale and $F_0=1$ then
$$
\tV(s,C,x)=S_0\leq \E_{s,x,i,1}[S^s_{t\wedge\s}]=\E_{s,x}[\int_0^{\T_0\wedge t} f(X_u)du +\tV(s,C,X^s_{\\T_0\wedge t})].
$$
Letting $t\uparrow\infty$ we see, by dominated convergence, that
$$
\tV(s,C,x)\leq \E_{s,x}[\int_0^{\T_0} f(X_t)dt].
$$
Minimising over admissible $s$ yields $\V(s,C,x,1)\geq \tV(s,C,x).$
Conversely, $\liminf \E_{\hs^n,x,i,1} [S^{\hs^n}_{\s}]=\tV(s,C,x)$ implies that $\V(s,C,x,1)\leq \tV(s,C,x)$, establishing equality.

Now suppose that $\s$ is admissible and  $(S^s_{t\wedge\s})_{t\geq 0}$ is a submartingale and $F_0=0$ then
$$
V(s,x,i)=S_0\leq \E_{x,i,0}[S^s_{t\wedge\T_1})]=\E_{x,i,0}[\int_0^{\T_1\wedge t} f(X^s_u)du +V(\rinf^s_t,X^s_{\\T_1\wedge t})].
$$
Letting $t\uparrow\infty$ we see that 
$$
\V(s,C,x,i)\leq \E_{x,i,0}[\int_0^\s f(X^s_t)dt].
$$
Since $s$ is arbitrary we deduce $\V(s,C,x,0)\geq \V(s,C,x,i).$
Conversely, $ \E_{x,i,0} [S^{\hs}_{\s}]=V(s,C,x,i)$ implies that $\V(s,C,x,0)\leq V(s,C,x,i)$, establishing equality.
\end{proof}

Now we establish our main result.
\begin{thm}\label{result}
The value function $\V$ is given by
$$
\V(s,C,x,f)=\begin{cases}
V(s,C,x):&f=0\\
\tV(s,C,x):&f=1,
\end{cases}
$$
\end{thm}
\begin{proof}We prove that $V$ and $\tV$ satisfy the conditions of Theorem \ref{sub}.
\begin{itemize}
\item[1.] Consider $S^s_t$.
Using the fact that
\eqn
N_t\defto\int_0^{t\wedge T_1}f(X_u)\diff u+\phi^s(X_{t\wedge T_1},1)=\E_{s,i}[\int_0^{T_1}f(X_u)\diff u|\F_{t\wedge T_1}]
\endeqn
is a $\P_{s,i}$-martingale,
we see that $S^s_{\cdot}$ is a $\P_{s,i}$-martingale on the stochastic interval $[[t,\T_{\rinf_t}]]$ if $F_t=0$, and hence
$$
dS_t=dN_t+V_i(\rinf_t,X_t)d\rinf_t,
$$
where $N$ is a local martingale.

So, to establish that $S$ is a submartingale on $[0,\T_1]$, since $\rinf$ is a decreasing process {\em which only decreases when $\rinf_t=X_t$}, it is sufficient to show that

$$
V^s_i(i,i)\le 0\text{ and that there exists $\hat s$ such that }V^{\hat s}_i(i,i)= 0.
$$

Differentiating with respect to $i$ yields
\begin{align*}
\frac{\partial}{\partial i}V(i,x) 
&=  - s'(i) \tilde{m}_f (i)+ \frac{\tilde{s} (x)}{\tilde{s} (i)} V'(i) + (V(i)-\kappa) \tilde{s} (x) \frac{s'(i)}{(\tilde{s} (i) )^2} \end{align*}
Setting $x=i$ implies
\begin{equation}\label{delta}
\D:=\frac{\partial}{\partial i}V(i,x)\Big |_{x=i} 
=  - s'(i) \tilde{m}_f (i)+  V'(i) + (V(i)-\kappa)  \frac{s'(i)}{\tilde{s} (i) } 
\end{equation}

Now $V^s(i)= \kappa + bc+\B^2(i)\Psi^*( r)=\kappa+\mf(\ell)\ts(i)+\beta^2(i)\Psi^*_\d(r)$ with $\d=\d(i)=\frac{s(\ell)}{\ts(i)}$ and $r=\frac{\tmf(i)\ts(i)}{\beta^2(i)}$, consequently
\begin{equation}\label{delta2}
{V^s}'(i)=-\mf(\ell)s'(i)+{\B^2}'\psd(r)+\B^2[{\psd}'(r)r'(i)+\d'(i)\frac{\partial}{\partial \d}\psd(r)]
\end{equation}

Observe that $\d'(i)=\frac{s(\ell)s'(i)}{\ts^2(i)}=\d\frac{s'(i)}{\ts(i)}$ while $r'(i)=-\frac{{\B^2}'}{\B^2}r-\frac{\ts(i)\mf'(i)+\tmf(i)s'(i)}{\B^2}$.

As is standard, ${\ps}'(x)={\Psi}'^{-1}(-x)$ and $\ps(x)=x{\Psi}'^{-1}(-x)+\Psi({\Psi}'^{-1}(-x))$ so, 
\begin{IEEEeqnarray*}{rl}
\frac{\partial}{\partial \d}\psd(r)&=\frac{\partial}{\partial \d}[r{\pd}'^{-1}(-r)+\pd({\pd}'^{-1}(-r))]\\
&=r\frac{\partial}{\partial \d}{\pd}'^{-1}(-r)+(\frac{\partial}{\partial \d}\pd)(({\pd}'^{-1}(-r)))+{\pd}'({\pd}'^{-1}(-r))\frac{\partial}{\partial \d}{\pd}'^{-1}\\
&=(\frac{\partial}{\partial \d}\pd)(({\pd}'^{-1}(-r))).
\end{IEEEeqnarray*}

Setting $R_\d=R={\ps}'(r)$, we see that
\[\frac{\partial}{\partial \d}\psd(r)=(\frac{\partial}{\partial \d}\pd)(R)=\frac{\Psi^2(R)\fd^{-1}(R)^2}{\fd^{-1}(R)^2-\d^2}=\Psi^2(R)(g(R)-\frac{1}{R}).
\]
Substituting in (\ref{delta}) and applying (\ref{delta2}) yields
\begin{IEEEeqnarray}{rl}\label{delta3}
\D&={\B^2}'[\psd(r)-rR]+\B^2\d'(i)\frac{\partial}{\partial \d}\psd(r)]-\frac{s'(i)}{\ts(i)}[\tmf(i)\ts(i)-\B^2(i)\ps(r)+\tmf(i)\ts(i)R]-\mf'(i)\ts(i)R\nonumber\\
&={\B^2}'[\psd(r)-rR]+\B^2\d'(i)\frac{\partial}{\partial \d}\psd(r)]-\frac{s'(i)}{\ts(i)}\B^2(i)[r(R+1)-\ps(r)]-\mf'(i)\ts(i)R\nonumber\\
&={\B^2}'[\psd(r)-rR]-\frac{s'(i)}{\ts(i)}\B^2(i)[r(R+1)-\ps(r)-\Psi^2(R)(g(R)-\frac{1}{R})]-\rho R\frac{\ts(i)}{s'(i)}.
\end{IEEEeqnarray}

Now 
$$
\ps(r)=rR+\Psi(R)\Rightarrow r(R+1)-\ps(r)=r-\Psi(R)=-\Psi'(R)-\Psi(R)
=\Psi^2(R)g(R),
$$
and thus
\begin{equation}
\D={\B^2}'\Psi(R)-\frac{s'(i)}{\ts(i)}\B^2(i)\frac{\Psi^2(R)}{R}-\rho R\frac{\ts(i)}{s'(i)}.
\end{equation}
If we now use the fact, noted in \cite{JH}, that
for 
$a,b\geq 0$,
\begin{equation}\label{ineq}
\inf_{x>0}(ax+\frac{b}{x})=2\sqrt{ab}\,\,\text{ and if }\,\,a,b>0\,\,\text{ this is attained at }x=\sqrt{\frac{b}{a}},
\end{equation}
we see that

$$\D\leq {\B^2}'\Psi(R)-2\sqrt{\B^2(i)\rho \Psi^2(R)}=2\B\Psi(R)(\B'(i)-\sqrt{\rho})=0,
$$
as required, with equality attained if 
\begin{equation}\label{opt}
\frac{s'(i)}{R\ts(i)}=\sqrt{\frac{\rho}{\B^2(i)\Psi^2(R)}}\Rightarrow \frac{s'(i)}{\ts(i)}
=\frac{(\ln\B)'(i)}{\Psi(R)}\text{ for }i\in(\ell,i_0).
\end{equation}

\item[2.] Note that on the stochastic interval $[[t\vee \T_1,\inf\{u\geq t:\;X^s\in[\ell, \rinf_{\T_1}]\}]]$ 
$S^s_t$ is a $\P_{s,x}$-martingale,  $N$ say, and is constant on $[[[t\vee \T_1,\inf\{u\geq t:\;X^s\not\in[\ell, \rinf_{\T_1}]\}]]$. It follows that, denoting $V'(x)-V'(x-)$ by $\Delta V'(x)$,
$$
dS^s_t=dN_t+\Delta V'(i)dL^i_t+\Delta V'(\ell)dL^\ell_t,
$$
(recall that $L^a_t$ denotes the local time of $X^s$ at the level $a$ by time $t$).

Now $\Delta V'(\ell)=0$ and $\Delta V'(i)=s^*(i)\tmf(i)>0$ so we conclude that, since $L^\ell$ is an increasing process, $S^s$ is a $\P_{s,x}$-submartingale on $[[\T_1,\s]]$.

Now, defining 
$$
\begin{cases}
(\hs^n)'(x)=s'&\text{ if }x\leq\ell\\
s'(l)\exp\bigl(2n(x-\ell)\bigr)&\text{ if }x\in[\ell,i)\\
s'(l)\exp\bigl(2n(i-\ell)\bigr)\frac{s'(x)}{s'(i)}&\text{ if }x\geq i,
\end{cases}
$$
it is easy to check that
$$
\liminf \E_{s,x,i,1} [S^{\hs^n}_{\s}]=\tV(x),
$$
as required.
\end{itemize}

All that remains is to show that $s$, where $\hs$ is given by (\ref{opt}), is admissible. 

Set $Y=\hs(X)$. The existence up to time $\T_{i_0}\wedge\T_{\ell}$ of a weak solution to 
$$
dY=-s'\circ \ts^{-1}(Y_t)dB_t
$$
follows easily from \cite{E} on observing that $\B$ and $\B'$ are strictly positive on $(\ell,i_0]$ and $\Psi$ is strictly positive on $(\g,\infty)$ and, for $\ell<i<i_0$,
$$
\tilde\hs(i)=\ts(i_0)\exp(\int_\ell^i\frac{\B'(y)}{\B(y)\Psi(R(y))}dy).
$$
We may then extend this to a solution for all $t$ using standard techniques since $s_0$ is, by assumption, admissible.
This implies that $\P_{s,x}$ (as in (\ref{adm}) exists establishing that $s$ is admissible. 
\end{proof}

\section{Concluding remarks and further problems}
We remark first that the optimally controlled process is not a diffusion nor even strong Markov.
To make it into one we  have to adopt an infinite dimensional statespace recording the choice of $s'$ at each level so far visited.

Although we have solved the minimisation problem when $C_0$ is of the form $[0,\ell)\cup(i,1]$ the problem remains open for other forms of constraint sets.
It is tempting to speculate that a similar approach to the one adopted here would, with much further work, yield a result when $C_0$ is of the form
$$
[0,\ell_1)\cup (i_2,\ell_2)\cup\cdots\cup (i_n,1]
$$
and $X_0\geq i_n$ but we have no suggestions as to the form of optimal controls for any other forms of the constraint set. A complete solution would be fascinating.

The corresponding network problem, where $X$ is a reversible Markov chain on a specified graph
is equally interesting and merits further exploration.


\end{document}